\documentclass[a4paper,11pt]{article}
\usepackage[top=2.5cm,bottom=2.5cm,left=2.2cm,right=2.2cm]{geometry}
\usepackage{amsfonts}
\usepackage{mathrsfs,amscd,amssymb,amsthm,amsmath,bm,graphicx,psfrag,subfigure,url,mathtools}
\usepackage{pict2e}
\usepackage{stfloats}
\usepackage{psfrag,amsmath}
\usepackage{tikz}
\usetikzlibrary{positioning,fit,calc,arrows.meta,backgrounds,shapes.geometric}
\usepackage{indentfirst}
\usepackage{hyperref}
\usepackage{bookmark}
\usepackage{enumerate}
\usepackage{latexsym,euscript,epic,eepic,color}
\usepackage{multirow}
\usepackage{multicol}
\usepackage{longtable}
\usepackage{adjustbox}
\usepackage[all]{xy}
\usepackage{setspace}
\usepackage{booktabs}
\usepackage{epstopdf}
\allowdisplaybreaks
\usepackage{authblk}
\usepackage{cite}

\makeatletter

\renewcommand{\@seccntformat}[1]{{\csname the#1\endcsname}{\normalsize .}\hspace{.5em}}
\makeatother

\usepackage{ifpdf}
\usepackage{indentfirst}

\def \[{\begin{equation}}
	\def \]{\end{equation}}

\newtheorem{thm}{Theorem}[section]

\newtheorem{claim}{Claim}
\newtheorem{lemma}[thm]{Lemma} 
\newtheorem{cor}[thm]{Corollary}
\newtheorem{problem}[thm]{Problem} 
\newtheorem{conj}{Conjecture}

\DeclareMathOperator{\comp}{comp}
\newcommand{\one}{\mathbf 1}

\begin{document}

\title{\bf Nonhamiltonian regular sublinear expanders}
\author{%
  \textbf{Chengli Li}%
  \thanks{School of Mathematical Sciences, South China Normal University,
  Guangzhou 510631, China. Email: \texttt{lichengli0130@126.com}.}
  \qquad
  \textbf{Yurui Tang}%
  \thanks{School of Mathematical Sciences, East China Normal University,
  Shanghai 200241, China. Email: \texttt{tyr2290@163.com}.}
}
\date{}
\maketitle
\vspace{-2.5em}

\begin{abstract}
Letzter, Methuku and Sudakov proved that sufficiently regular sublinear expanders contain nearly spanning cycles and paths. Montgomery subsequently conjectured that every regular sublinear expander of sufficiently large degree is hamiltonian. Recently, Chen, Liu, Wei and Yang disproved this conjecture by constructing $d$-regular sublinear expanders of order $n$ with $d=(1/2+o(1))\log^2 n$ and small circumference. They further posed the problem of determining whether, for every fixed $\varepsilon>0$, there exists a constant $C=C(\varepsilon)$ such that every sufficiently large $n$-vertex $d$-regular $(\varepsilon,d)$-expander with $d\ge C\log^2 n$ is hamiltonian. We answer this problem negatively by constructing, for every sufficiently small fixed $\varepsilon>0$, infinitely many nonhamiltonian $d$-regular $(\varepsilon,d)$-expanders with degrees far above the $\log^2 n$ scale.

\smallskip
\noindent\textbf{Keywords:} Hamilton cycle, sublinear expander, regular graph, Ramanujan graph, toughness

\noindent\textbf{2020 Mathematics Subject Classification:} 05C45, 05C48, 05C50
\end{abstract}

\section{Introduction}

All graphs in this paper are finite, simple and undirected. For a graph $G$ and a set $X\subseteq V(G)$, the \emph{external neighborhood} of $X$ is
$N_G(X)=\{v\in V(G)\setminus X: v\text{ has a neighbor in }X\}$.
A \emph{Hamilton cycle} is a cycle containing every vertex, and a graph is
\emph{hamiltonian} if it has a Hamilton cycle. The \emph{circumference}
$c(G)$ is the maximum order of a cycle in $G$.

Following Chen, Liu, Wei and Yang~\cite{ChenLiuWeiYang}, for $\varepsilon>0$ and
$k>0$, define
$$
\rho(x;\varepsilon,k)=
\begin{cases}
0, & x<k/5,\\[2mm]
\displaystyle\frac{\varepsilon}{\log^2(15x/k)}, & x\ge k/5.
\end{cases}
$$
All logarithms are natural. A graph $G$ is an $(\varepsilon,k)$-expander if $|N_G(X)|\ge \rho(|X|;\varepsilon,k)|X|$ for every set $X\subseteq V(G)$ with $k/2\le |X|\le |V(G)|/2$.

Hamiltonicity is one of the central global properties in graph theory. Classical theorems of Dirac and Ore force a Hamilton cycle through large minimum degree or large degree sums \cite{Dirac,Ore}. The theorem of Chv\'atal and Erd\H{o}s gives another dense-graph condition by comparing connectivity and independence number \cite{ChvatalErdos}. In sparse graphs, expansion often replaces density. This point of view appears in work on pseudorandom graphs and deterministic expanders, including the results of Krivelevich and Sudakov \cite{KrivelevichSudakov}, Hefetz, Krivelevich and Szab\'o \cite{HefetzKrivelevichSzabo}, Dragani\'c, Montgomery, Munh\'a Correia, Pokrovskiy and Sudakov \cite{DraganicEtAl}, and 
Brada\v{c} and Janzer \cite{BradavcJanzer}.

Sublinear expansion is weaker than the usual constant-factor expansion. It was developed by Koml\'os and Szemer\'edi in their work on topological cliques \cite{KomlosSzemerediI,KomlosSzemerediII}; see also the survey of Letzter \cite{LetzterSurvey} and the recent survey of Montgomery \cite{MontgomerySurvey}. Its weakness makes it possible to find sublinear expanders inside broad classes of sparse graphs, while its strength is still sufficient for many embedding problems. Letzter, Methuku and Sudakov proved that sufficiently regular sublinear expanders contain nearly spanning cycles and paths \cite{LetzterMethukuSudakov}. Motivated by this result, Montgomery proposed the following conjecture.
\begin{conj}[Montgomery\cite{MontgomerySurvey}]\label{Montgomery}
For every $\varepsilon>0$, there exists a constant $d_0$ such that, for every $d\ge d_0$, every $d$-regular $(\varepsilon,d)$-expander is hamiltonian.
\end{conj}
Chen, Liu, Wei and Yang \cite{ChenLiuWeiYang} disproved Conjecture~\ref{Montgomery} by constructing regular sublinear expanders with small circumference.
\begin{thm}[Chen--Liu--Wei--Yang \cite{ChenLiuWeiYang}]
Let $0<\eta<1/2$ and $0<\varepsilon_2<1/20$. There exists
$\varepsilon=\varepsilon(\eta,\varepsilon_2)>0$ such that, for every
fixed $0<\varepsilon_1<\varepsilon$ and infinitely many integer pairs $(d,n)$, there is a
$d$-regular $n$-vertex graph $G$ such that
\begin{enumerate}[\rm (1)]
 \item $G$ is an $(\varepsilon_1,\varepsilon_2d)$-expander,
 \item $c(G)<\eta n$, and
 \item $d=(1/2+o_d(1))\log^2 n$.
\end{enumerate}
\end{thm}
Chen, Liu, Wei and Yang constructed regular sublinear expanders of degree $(1/2+o(1))\log^2 n$ and they further posed the problem of determining whether a degree threshold of order $\log^2 n$ is already sufficient to force hamiltonicity.

\begin{problem}[{Chen--Liu--Wei--Yang \cite[Problem~5.1]{ChenLiuWeiYang}}]\label{prob}
For every $\varepsilon>0$, does there exist a constant $C=C(\varepsilon)>0$ such that every sufficiently large $n$-vertex $d$-regular $(\varepsilon,d)$-expander with $d\ge C\log^2 n$ contains a Hamilton cycle?
\end{problem}

Our answer is negative. The construction theorem below is stronger than the expansion demanded in Problem~\ref{prob} because it gives constant-factor vertex expansion for every nonempty set of size at most half the order.

\begin{thm}\label{main}
There is an absolute integer $d_0$ such that, for every integer $d\ge d_0$ is
divisible by $4$ and every integer $r\ge1$, there is a simple nonhamiltonian $d$-regular
graph $G=G(d,r)$ of order $2^{r+1}d+2d+1$ such that every nonempty set
$X\subseteq V(G)$ with $|X|\le |V(G)|/2$ satisfies
$|N_G(X)|\ge |X|/72$.
\end{thm}

Let $\varepsilon_0=\log^2(15/2)/72$. Theorem~\ref{main} immediately yields the following consequence.

\begin{cor}\label{answer}
For every fixed $\varepsilon$ with $0<\varepsilon\le\varepsilon_0$,
Problem~\ref{prob} has a negative answer. More precisely, there are
infinitely many nonhamiltonian $n$-vertex $d$-regular
$(\varepsilon,d)$-expanders for which $d=o(n)$ and
$d/\log^2 n\to\infty$.
\end{cor}

The remainder of the paper is organized as follows.
Section~\ref{pre} collects the spectral and toughness facts used later.
Section~\ref{proof} gives the construction and proves
Theorem~\ref{main} and Corollary~\ref{answer}.

\section{Preliminaries}\label{pre}

For a graph $F$ of order $N$, let $A_F$ denote its adjacency matrix, and
write its adjacency eigenvalues in nonincreasing order as
$$
\lambda_1(F)\ge \lambda_2(F)\ge\cdots\ge\lambda_N(F).
$$
If $F$ is connected and $d$-regular, then $\lambda_1(F)=d$ and this
eigenvalue is simple. If, in addition, $F$ is bipartite, then its spectrum is
symmetric about zero and $\lambda_N(F)=-d$. The two eigenvalues $d$ and $-d$
are called the \emph{trivial eigenvalues}. A connected $d$-regular bipartite
graph $F$ is \emph{Ramanujan} if
$$
\max_{2\le j\le N-1}|\lambda_j(F)|\le 2\sqrt{d-1}.
$$
For disjoint sets $U,W\subseteq V(F)$, we denote by $e_F(U,W)$ the number of
edges of $F$ with one endpoint in $U$ and the other in $W$.

We obtain the base graph required below directly from
\cite[Lemma~2.1]{ChenLiuWeiYang}. Taking $c=d$ and $t=r$ in that lemma gives
a $d$-regular bipartite Ramanujan graph with both bipartition classes of size
$d2^r$. The construction there yields a simple graph. It is also connected,
since otherwise $d$ would be a nontrivial adjacency eigenvalue, contradicting
$2\sqrt{d-1}<d$.

\begin{lemma}\label{Ramanujan}
For every pair of integers $d\ge3$ and $r\ge0$, there is a connected simple
$d$-regular bipartite Ramanujan graph $B$ with bipartition $(L,R)$ such that
$|L|=|R|=d\,2^r$.
\end{lemma}

We shall use the elementary hamiltonian obstruction arising from toughness.
For $S\subseteq V(F)$, let $F-S:=F[V(F)\setminus S]$, and let $\comp(F-S)$ denote
the number of components of $F-S$. A graph $F$ is \emph{$1$-tough} if
$\comp(F-S)\le |S|$ for every vertex cut $S$. It is well known that every
hamiltonian graph is $1$-tough.

The only spectral calculation needed later is the following standard cut
estimate; see, for example, \cite{Chung,HooryLinialWigderson}.

\begin{lemma}\label{cut}
Let $F$ be a connected $d$-regular graph of order $N$. For every set
$U\subseteq V(F)$,
$$
e_F(U,V(F)\setminus U)\ge
\bigl(d-\lambda_2(F)\bigr)|U|\left(1-\frac{|U|}{N}\right).
$$
\end{lemma}

\begin{proof}
Let $L_F=dI-A_F$ be the Laplacian matrix of $F$. Let $s=|U|$ and decompose
the characteristic vector of $U$ as $\one_U=(s/N)\one+z$, where
$z\perp\one$. Since $L_F\one=0$ and every eigenvalue of $L_F$ on
$\one^\perp$ is at least $d-\lambda_2(F)$, we have
$$e_F(U,V(F)\setminus U)
  =\one_U^{\mathsf T}L_F\one_U
  =z^{\mathsf T}L_Fz
  \ge \bigl(d-\lambda_2(F)\bigr)\|z\|^2.$$
Because $\|z\|^2=s(1-s/N)$, the desired inequality follows.
\end{proof}

We now turn the spectral gap into the two neighborhood estimates used in the
construction.

\begin{lemma}\label{expansion}
Let $d\ge64$, and let $B$ be a $d$-regular bipartite Ramanujan graph of order
$2m$. If $\emptyset\ne U\subseteq V(B)$ and $|U|\le m$, then
$$
|N_B(U)|\ge \frac{|U|}{4}
\qquad\text{and}\qquad
|N_B(U)|\ge \frac d8.
$$
\end{lemma}

\begin{proof}
Put $s=|U|$. Since $B$ is Ramanujan and $d\ge64$, its second largest
adjacency eigenvalue satisfies
$\lambda_2(B)\le2\sqrt{d-1}\le d/4$. The inequality $s\le m$ gives
$1-s/(2m)\ge1/2$. Lemma~\ref{cut} therefore yields
$$
e_B(U,V(B)\setminus U)
\ge \frac{3d}{4}s\left(1-\frac{s}{2m}\right)
\ge\frac{3ds}{8}.
$$
Every vertex in $N_B(U)$ is incident with at most $d$ edges of this cut.
Hence $|N_B(U)|\ge3s/8\ge s/4$.

It remains to obtain a lower bound independent of $s$. When $s\ge d/2$, the
preceding estimate gives $|N_B(U)|\ge s/4\ge d/8$. When $s<d/2$, choose
$v\in U$. At most $s-1$ neighbors of $v$ lie in $U$, so more than $d/2$
neighbors of $v$ lie outside $U$. All of them belong to $N_B(U)$, and hence
$|N_B(U)|\ge d/8$.
\end{proof}

\section{Proofs of the main results}\label{proof}

\subsection{The construction and its basic properties}\label{construction}

Fix an integer $d\ge72$ divisible by $4$ and an integer $r\ge1$, and put
$m=d\,2^r$. Let $B=(L,R;E(B))$
be a connected simple $d$-regular bipartite Ramanujan graph with
$|L|=|R|=m$, whose existence follows from
Lemma~\ref{Ramanujan}.

Choose a vertex $u\in L$ and partition its neighborhood into two sets of
equal size:
$$
N_B(u)=R_1\cup R_2,
\qquad |R_1|=|R_2|=\frac d2.
$$
For each $i\in\{1,2\}$, take a copy $K_i$ of $K_{d+1}$, with
$V(K_1),V(K_2)$ and $V(B)$ pairwise disjoint. Choose a matching
$M_i\subseteq E(K_i)$ of size $d/4$ and let
$
H_i:=K_i-M_i.
$
Let $J_i\subseteq V(H_i)$ be the set of endpoints of the edges in $M_i$.
Thus $|J_i|=d/2$. We call the vertices in $J_i$ the \emph{ports} of $H_i$,
and fix a bijection $\varphi_i:J_i\to R_i$.

Define $G=G(d,r)$ by
$$
V(G)
=\bigl(V(B)\setminus\{u\}\bigr)
 \cup V(H_1)\cup V(H_2)
$$
and
$$
\begin{aligned}
E(G)
=E(B-u)\cup E(H_1)\cup E(H_2)
 \cup
 \bigl\{p\varphi_i(p):
 i\in\{1,2\},\ p\in J_i\bigr\}.
\end{aligned}
$$
In other words, $u$ is replaced by the two graphs $H_1,H_2$, and each
$R_i$ is joined by a matching to the ports of $H_i$.

Figure~\ref{fig} illustrates this construction. For each
$i\in\{1,2\}$, the line joining $R_i$ to $H_i$ represents the matching
induced by $\varphi_i$ between $R_i$ and the port set $J_i$.

\begin{figure}[ht]
\centering
\begin{tikzpicture}[
  font=\small,
  box/.style={draw,rounded corners,minimum width=25mm,minimum height=16mm,align=center},
  bigbox/.style={draw,ellipse,minimum width=28mm,minimum height=46mm,align=center},
  matching/.style={thick},
  every node/.style={inner sep=3pt}
]
\node[bigbox] (L) {$L\setminus\{u\}$};
\node[bigbox,right=32mm of L] (R) {$R$\\separator};
\node[box,right=35mm of R,yshift=15mm] (H1) {$H_1=K_{d+1}-M_1$};
\node[box,right=35mm of R,yshift=-15mm] (H2) {$H_2=K_{d+1}-M_2$};

\draw[thick] (L.east) -- node[above,sloped] {edges of $B-u$} (R.west);
\draw[matching] ($(R.east)+(0,10mm)$) -- node[above,sloped] {$\varphi_1$} (H1.west);
\draw[matching] ($(R.east)+(0,-10mm)$) -- node[below,sloped] {$\varphi_2$} (H2.west);
\end{tikzpicture}
\caption{A schematic illustration of the graph $G(d,r)$.}
\label{fig}
\end{figure}
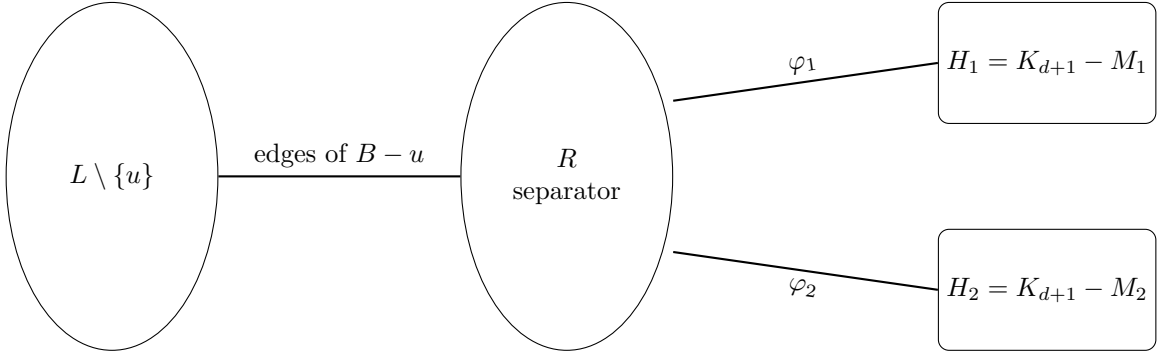

\begin{lemma}\label{nonham}
For every integer $d\ge72$ divisible by $4$ and every integer $r\ge1$, the
graph $G=G(d,r)$ is $d$-regular and nonhamiltonian with
$|V(G)|=2m+2d+1=2^{r+1}d+2d+1$.
\end{lemma}

\begin{proof}
Every vertex of $L\setminus\{u\}$ retains its $d$ incident edges. A vertex in
$R_i$ loses its edge to $u$ and receives exactly one new edge from its unique
preimage under $\varphi_i$, while every vertex in
$R\setminus N_B(u)$ is unchanged.

Inside $K_i$, every vertex initially has degree $d$. A non-port vertex is not
incident with $M_i$ and therefore has degree $d$ in $H_i$. Each port is an
endpoint of exactly one edge of $M_i$, so it has degree $d-1$ in $H_i$ and
one additional neighbor in $R_i$. Hence every vertex of $G$ has degree $d$.
The graph is simple because the three vertex sets in the construction are
disjoint, the constituent graphs are
simple, and each $\varphi_i$ is a bijection. The order is
$$
|V(G)|=(2m-1)+2(d+1)=2m+2d+1.
$$

After deleting $R$, all $m-1$ vertices of $L\setminus\{u\}$ are isolated.
The graphs $H_1$ and $H_2$ form two further components (each is connected,
since it is obtained from a complete graph by deleting a matching). Therefore
$$
\comp(G-R)=(m-1)+2=m+1>|R|=m.
$$
Thus $G$ is not $1$-tough and therefore cannot be hamiltonian.
\end{proof}

\subsection{Proof of Theorem~\ref{main}}

Set $d_0=72$, and let $d\ge d_0$ be divisible by $4$ and $r\ge1$.
Let $G=G(d,r)$ be the graph constructed in Section~\ref{construction}.
Lemma~\ref{nonham} shows that $G$ is simple and
$d$-regular, has the order asserted in the theorem, and is not hamiltonian.
It remains to prove the vertex-expansion assertion.

Fix a nonempty set $X\subseteq V(G)$, and let
$
x:=|X|\le|V(G)|/2.$
Set
$$
T:=N_G(X),\qquad b:=|T|,\qquad
Z:=V(G)\setminus(X\cup T).
$$
By the definition of the external neighborhood, there is no edge between
$X$ and $Z$.

We separate the vertices inherited from the base graph from those in the two
new blocks by putting
$$
Q:=V(B)\setminus\{u\},
\qquad W:=V(H_1)\cup V(H_2).
$$
Thus $|Q|=2m-1$ and $|W|=2d+2$. Finally, let
$$
A:=X\cap Q
\qquad\text{and}\qquad
D:=Z\cap Q.
$$

Every edge of $B$ with both endpoints in $Q$ is retained in $G$. Since there
is no edge between $X$ and $Z$, it follows that
\begin{equation*}
N_B(A)\subseteq (T\cap Q)\cup\{u\}
\quad\text{and}\quad
N_B(D)\subseteq (T\cap Q)\cup\{u\}.
\end{equation*}
It follows from Lemma~\ref{expansion} that, for
$Y\in\{A,D\}$,
\begin{equation}\label{boundary}
0<|Y|\le m
\quad\Longrightarrow\quad
b+1\ge \max\left\{\frac{|Y|}{4},\frac d8\right\}.
\end{equation}

We shall also repeatedly use the half-order assumption that $x=|X|\le |V(G)|/2$. Since
$|V(G)|=2m+2d+1$, it gives
\begin{equation}\label{x}
x\le m+d.
\end{equation}

\begin{claim}\label{small}
For the set $X$ fixed above, we have $b\ge d/9$.
\end{claim}

\begin{proof}
Suppose first that $A\ne\emptyset$. If $|A|\le m$, then
\eqref{boundary} gives $b\ge d/8-1\ge d/9$, where the last
inequality uses $d\ge72$.

Assume now that $|A|>m$. Since $A$ and $D$ are disjoint subsets of the
$(2m-1)$-vertex set $Q$, we have $|D|<m$. If $D\ne\emptyset$, applying
\eqref{boundary} to $D$ again gives $b\ge d/9$. If
$D=\emptyset$, then $Q\setminus A\subseteq T$. Since $|A|\le x$,
\eqref{x} gives
$$
b\ge |Q|-|A|\ge 2m-1-(m+d)=m-d-1.
$$
The assumption $r\ge1$ implies $m=d\,2^r\ge2d$, and hence
$b\ge d-1\ge d/9$.

It remains to consider $A=\emptyset$. Then $X\subseteq W$. Choose
$i\in\{1,2\}$ such that
$X_i:=X\cap V(H_i)$ is nonempty. If $|X_i|=1$, its unique vertex has at
least $d-1$ neighbors in $H_i\setminus X_i$, so $b\ge d-1$.

Suppose that $|X_i|\ge2$. Every vertex of $H_i$ has at most one nonneighbor
within $H_i$. Thus every vertex of $V(H_i)\setminus X_i$ has a neighbor in
$X_i$ and belongs to $T$. Moreover, the vertices in $X_i\cap J_i$ have
distinct neighbors in $R_i$. Since $A=\emptyset$, all these neighbors lie
outside $X$. The two groups of neighbors lie in disjoint vertex sets, so
$$
b\ge |V(H_i)\setminus X_i|+|X_i\cap J_i|
   =d+1-|X_i\setminus J_i|\ge d/2.
$$
Here the last inequality follows from
$|V(H_i)\setminus J_i|=d/2+1$. This proves the claim.
\end{proof}

\begin{claim}\label{large}
If $x\ge8d$, then $b\ge x/10$.
\end{claim}

\begin{proof}
Since $|W|=2d+2$ and $x\ge8d$, the set $A=X\cap Q$ is nonempty.

Suppose first that $|A|\le m$. By \eqref{boundary},
$$
b\ge\frac{|A|}{4}-1
\ge\frac{x-2d-2}{4}-1
\ge\frac{x}{10},
$$
where the last inequality follows from $x\ge8d$ and $d\ge72$.

We may therefore assume that $|A|>m$, which implies $|D|<m$. If
$D=\emptyset$, then $Q\setminus A\subseteq T$. Using $|A|\le x$ and
\eqref{x}, we obtain
$$
b\ge 2m-1-x\ge x-2d-1\ge x/10.
$$

Finally, suppose that $D\ne\emptyset$. By
\eqref{boundary}, $b+1\ge |D|/4$. The sets $A$, $D$ and
$T\cap Q$ partition $Q$, and hence
$$
|D|=2m-1-|A|-|T\cap Q|\ge2m-1-x-b\ge x-2d-1-b,
$$
where the last inequality again follows from \eqref{x}. Therefore
$$
4b+4\ge x-2d-1-b,
$$
and hence
$$
b\ge\frac{x-2d-5}{5}\ge\frac{x}{10}.
$$
The final inequality uses $x\ge8d$ and $d\ge72$. This completes the proof.
\end{proof}

If $x<8d$, Claim~\ref{small} gives
$b\ge d/9>x/72$. If $x\ge8d$, Claim~\ref{large} gives
$b\ge x/10\ge x/72$. Thus every nonempty set $X\subseteq V(G)$ with
$|X|\le |V(G)|/2$ satisfies
$
|N_G(X)|\ge\frac{|X|}{72}.
$
Together with Lemma~\ref{nonham}, this proves
Theorem~\ref{main}.

\subsection{Proof of Corollary~\ref{answer}}

Let $G=G(d,r)$ be a graph from Theorem~\ref{main}, and let
$n=|V(G)|$. If $d/2\le |X|\le n/2$, then
$15|X|/d\ge15/2$. Since the function $\log^2(15x/d)$ is increasing on the interval of $d/2\le x\le n/2$, every $0<\varepsilon\le\varepsilon_0$ satisfies
$$
\rho(|X|;\varepsilon,d)\le\frac{\varepsilon}{\log^2(15/2)}\le\frac1{72}.
$$
Theorem~\ref{main} therefore implies that $G$ is an
$(\varepsilon,d)$-expander.

For the required family, let $d$ tend to infinity through multiples of $4$ and
choose $r=\lfloor d^{1/4}\rfloor$. Since
$n=2^{r+1}d+2d+1=2^{r+1}d(1+o(1))$, we have
$$
\log n=r\log 2+\log(2d)+o(1)=(\log 2+o(1))d^{1/4}.
$$
It follows that
$d/\log^2 n=((\log 2)^{-2}+o(1))d^{1/2}\to\infty$. Moreover,
$d/n\to0$ because $r\to\infty$. Thus $d=o(n)$, and these graphs answer
Problem~\ref{prob} negatively. This proves Corollary~\ref{answer}.
\section*{\normalsize Acknowledgement}  

The authors thank Professor Xingzhi Zhan for his constant support and guidance. This research was supported by the NSFC Grant No. 12271170
and by Science and Technology Commission of Shanghai Municipality Grant No. 22DZ2229014.

\section*{\normalsize Declaration on the use of AI}

ChatGPT was used to assist in developing and checking the
constructions and proofs. The authors independently verified all mathematical
arguments, wrote the paper, and take full responsibility for its content.

{\small
\renewcommand{\baselinestretch}{1.0}\selectfont

}

\end{document}